\documentclass[graybox]{svmult}

\usepackage{type1cm}        
\usepackage{makeidx}         
\usepackage{graphicx}        
\usepackage{multicol}        
\usepackage[bottom]{footmisc}

\usepackage{newtxtext}       %
\usepackage{newtxmath}       

\makeindex             

\begin{document}

\title*{Lower bounds for the number of random bits in Monte Carlo algorithms}
\author{Stefan Heinrich}
\institute{ \at Department of Computer Science,
	University of Kaiserslautern,
	D-67653 Kaiserslautern, Germany
	\email{heinrich@informatik.uni-kl.de}}
%
%
\maketitle

\abstract*{We continue the study of restricted Monte Carlo algorithms in a general setting. Here we show a lower bound  for minimal errors in the setting with finite restriction in terms of deterministic minimal errors. This generalizes a result of Heinrich, Novak, and Pfeiffer, 2004 to the adaptive setting. As a consequence, the lower bounds on the number of random bits from that paper also hold in this setting. We also derive a lower bound on the number of needed bits for integration of Lipschitz functions over the Wiener space, complementing a result of  Giles, Hefter, Mayer, and  Ritter, 2019. }

\abstract{We continue the study of restricted Monte Carlo algorithms in a general setting. Here we show a lower bound  for minimal errors in the setting with finite restriction in terms of deterministic minimal errors. This generalizes a result of \cite{HNP04} to the adaptive setting. As a consequence, the lower bounds on the number of random bits from \cite{HNP04} also hold in this setting. We also derive a lower bound on the number of needed bits for integration of Lipschitz functions over the Wiener space, complementing a result of  \cite{GHMR19b}. }

\section{Introduction}\label{sec:1}
Restricted Monte Carlo algorithms were considered in  \cite{Nov85,Nov88,TW92,HNP04,NP04,GYW06,YH08,GHMR19a,GHMR19b, GHMR20}. Restriction usually means that the algorithm has access only to random bits or to random variables with finite range.
Most of these papers on restricted randomized algorithms consider the non-adaptive case. Only \cite{GHMR19b} includes adaptivity, but considers a  class of algorithms where each information call is followed by one random bit call.

A general definition restricted Monte Carlo algorithms was given in \cite{Hei20}. It extends the previous notions in two ways: Firstly, it includes full adaptivity, and secondly, it includes models in which the algorithms have access to an arbitrary, but fixed set of random variables, for example, uniform distributions on $[0,1]$.
In \cite{Hei20} the relation of restricted to unrestricted randomized algorithms was studied. In particular, it was shown that for each such restricted setting there is a computational problem that can be solved in the unrestricted randomized setting but not under the restriction. 

The aim of the present paper is to continue the study of the restricted setting. The main result is a lower bound  for minimal errors in the setting with a finite restriction in terms of deterministic minimal errors. This generalizes a corresponding result from \cite{HNP04}, see Proposition\ 1 there, to the adaptive setting with arbitrary finite restriction. The formal proof in this setting is technically more involved. As a consequence 
the lower bounds on the number of random bits from \cite{HNP04} also hold in this setting. Another corollary concerns integration of Lipschitz functions over the Wiener space \cite{GHMR19b}. It shows that the number of random bits used in the algorithm from \cite{GHMR19b} is optimal, up to logarithmic factors. 

\section{Restricted randomized algorithms in a general setting} 
\label{sec:2}
We work in the framework of information-based complexity theory (IBC) \cite{Nov88,TWW88}, using specifically the general approach from \cite{Hei05a, Hei05b}. We recall the notion of a restricted randomized algorithm as recently introduced in \cite{Hei20}. This section is kept general, for specific examples illustrating this setup we refer to the integration problem considered in \cite{Hei20} as well as to the problems studied in Section \ref{sec:4}.

We consider an abstract  numerical problem 
\begin{equation}
\label{I4}
\mathcal{P}=(F,G,S,K,\Lambda),
\end{equation}
where $F$ and $K$ are a non-empty sets, 
$G$ is a Banach space,  $S$ a mapping from $F$ to $G$, and $\Lambda$ a nonempty set of mappings from $F$ to $K$.  The operator $S$ is understood to be the  solution operator that sends the input  $f\in F$ to the exact solution $S(f)$ and $\Lambda$ is the set of information functionals about the input $f\in F$ that can be exploited by an algorithm. 

A probability space with  access restriction is a tuple 
\begin{equation}
\label{B7}
\mathcal{R}=\big((\Omega,\Sigma,{\mathbb P}),K',\Lambda'\big),
\end{equation}
with $(\Omega,\Sigma,{\mathbb P})$  a probability space, $K'$ a non-empty set, and $\Lambda'$ a non-empty set of mappings from $\Omega$ to $K'$. Define
$$
\bar{K}=K\dot{\cup} K', \quad\bar{\Lambda}=\Lambda\dot{\cup}\Lambda',
$$ 
where $\dot{\cup}$ is the disjoint union, and for $\lambda\in\bar{\Lambda}$, $f\in F$, $\omega\in\Omega$ we set 
$$
\lambda(f,\omega)=\left\{\begin{array}{lll}
&\lambda(f) &\mbox{if} \quad \lambda\in\Lambda   \\
& \lambda(\omega) &\mbox{if} \quad \lambda\in\Lambda'.  \\
\end{array}
\right. 
$$
An $\mathcal{R}$-restricted  randomized algorithm for problem $\mathcal{P}$ is a tuple
$$
A=((L_i)_{i=1}^\infty, (\tau_i)_{i=0}^\infty,(\varphi_i)_{i=0}^\infty)
$$
such that 
$L_1\in\bar{\Lambda}$, $\tau_0\in\{0,1\}$, $\varphi_0\in G$,
and for $i\in {\mathbb{N}}$
\begin{equation*}
L_{i+1} : \bar{K}^i\to \bar{\Lambda},\quad
\tau_i:  \bar{K}^i\to \{0,1\},\quad
\varphi_i:  \bar{K}^i\to G 
\end{equation*}
are any mappings.
Given $f\in F$ and $\omega\in\Omega$, we define $(\lambda_i)_{i=1}^\infty$ with $\lambda_i\in \bar{\Lambda}$ 
as follows:
\begin{eqnarray}
\lambda_1=L_1, \quad
\lambda_i=L_i(\lambda_1(f,\omega),\dots,\lambda_{i-1}(f,\omega))\quad(i\ge 2)
\label{A1}.
\end{eqnarray}
If $\tau_0=1$, we define 
$$
{\rm card}_{\bar{\Lambda}}(A,f,\omega)={\rm card}_\Lambda(A,f,\omega)={\rm card}_{\Lambda'}(A,f,\omega)=0.
$$ 
If $\tau_0=0$, let ${\rm card}_{\bar{\Lambda}}(A,f,\omega)$ be
the first integer $n\ge 1$ with 
$$
\tau_n(\lambda_1(f,\omega),\dots,\lambda_n(f,\omega))=1, 
$$
if there is such an $n$. If $\tau_0=0$ and no such $n\in {\mathbb{N}}$ exists, 
put ${\rm card}_{\bar{\Lambda}}(A,f,\omega)=\infty$. 
Furthermore, set
\begin{eqnarray*}
	{\rm card}_\Lambda(A,f,\omega)&=&|\{k\le {\rm card}_{\bar{\Lambda}}(A,f,\omega):\lambda_k\in \Lambda\}|
	\\
	{\rm card}_{\Lambda'}(A,f,\omega)&=&|\{k\le {\rm card}_{\bar{\Lambda}}(A,f,\omega):\lambda_k\in \Lambda'\}|.
\end{eqnarray*}
We have
${\rm card}_{\bar{\Lambda}}(A,f,\omega)={\rm card}_\Lambda(A,f,\omega)+{\rm card}_{\Lambda'}(A,f,\omega) $. The output $A(f,\omega)$ of algorithm $A$ at input $(f,\omega)$ is defined as
\begin{equation}
\label{A3}
A(f,\omega)=\left\{\begin{array}{lll}
\varphi_0  & \quad\mbox{if} \quad {\rm card}_{\bar{\Lambda}}(A,f,\omega)\in \{0,\infty\} \\[.2cm]
\varphi_n(\lambda_1(f,\omega),\dots,\lambda_n(f,\omega))  &\quad\mbox{if} \quad 1\le {\rm card}_{\bar{\Lambda}}(A,f,\omega)=n<\infty. 
\end{array}
\right.
\end{equation}
Thus, a restricted randomized algorithm  can access the randomness of $(\Omega,\Sigma,\mathbb{P})$ only through the functionals $\lambda(\omega)$ for $\lambda\in \Lambda'$. 

The set of all $\mathcal R$-restricted randomized algorithms for $\mathcal P$ is denoted by $\mathcal{A}^{\rm ran}(\mathcal{P},\mathcal{R})$.
Let $\mathcal{A}_{\rm meas}^{\rm ran}(\mathcal{P},\mathcal{R})$ be the subset of those $A\in\mathcal{A}^{\rm ran}(\mathcal{P},\mathcal{R})$ with the following properties:  For each $f\in F$ the mappings 
$$
\omega\to{\rm card}_\Lambda(A,f,\omega)\in {\mathbb{N}}_0\cup\{\infty\},\quad \omega\to{\rm card}_{\Lambda'}(A,f,\omega)\in {\mathbb{N}}_0\cup\{\infty\}
$$  
(and hence $\omega\to {\rm card}_{\bar{\Lambda}}(A,f,\omega)$) are $\Sigma$-measurable 
and the mapping 
$
\omega\to A(f,\omega)\in G
$ 
is $\Sigma$-to-Borel measurable and ${\mathbb P}$-almost surely separably valued, the latter meaning that there is a separable subspace $G_f\subset G$ such that ${\mathbb P}(\{\omega\in\Omega:\, A(f,\omega)\in G_f\})=1$.
The error of $A\in \mathcal{A}_{\rm meas}^{\rm ran}(\mathcal{P},\mathcal{R})$ is defined as 
\begin{equation}
\label{E3}
e(\mathcal{P},A)=\sup_{f\in F}{\mathbb E}\,\|S(f)-A(f,\omega)\|_G.
\end{equation}
Given $n,k\in{\mathbb{N}}_0$, we define $\mathcal{A}_{n,k}^{\rm ran}(\mathcal{P},\mathcal{R})$ to be the set of those
$A\in \mathcal{A}_{\rm meas}^{\rm ran}(\mathcal{P},\mathcal{R})$ satisfying for each $f\in F$ 
$$
{\mathbb E}\,{\rm card}_\Lambda (A,f,\omega)\le n, \quad {\mathbb E}\,{\rm card}_{\Lambda'} (A,f,\omega)\le k.
$$
The $(n,k)$-th minimal  $\mathcal{R}$-restricted  randomized error of $S$ is defined as
\begin{equation}
\label{A7}
e_{n,k}^{\rm ran } (\mathcal{P},\mathcal{R})=\inf_{A\in\mathcal{A}_{n,k}^{\rm ran }(\mathcal{P},\mathcal{R}) }  
e(\mathcal{P},A).
\end{equation}

Special cases are the following: An access restriction $\mathcal R$ is called finite, if 
\begin{eqnarray}
\label{B6}
|K'|<\infty,\quad \lambda^{-1}(\{u\})\in \Sigma \quad(\lambda'\in\Lambda', u\in K').
\end{eqnarray}
In this case any $\mathcal R$-restricted randomized algorithm satisfies the following. For fixed $i\in {\mathbb{N}}_0$ and $f\in F$ the functions (see \eqref{A1})
$$
\omega\to L_i(\lambda_1(f,\omega),\dots,\lambda_{i-1}(f,\omega))\in \overline{\Lambda}, \quad \omega\to \lambda_i(f,\omega)\in \overline{K}
$$
take finitely many values and are $\Sigma$-to-$\Sigma_0(\overline{\Lambda})$-measurable (respectively  $\Sigma$-to-$\Sigma_0(\overline{K})$-measurable), where $\Sigma_0(M)$ denotes the $\sigma$-algebra generated by the finite subsets of a set $M$. This is readily checked by induction. It follows that the mapping 
$$
\omega \to\tau_i(\lambda_1(f,\omega),\dots,\lambda_i(f,\omega))\in \{0,1\}
$$
is measurable and 
$$
\omega \to\varphi_i(\lambda_1(f,\omega),\dots,\lambda_i(f,\omega))\in G
$$
takes only finitely many values and is $\Sigma$-to-Borel-measurable. Consequently, for each $f\in F$ the functions  ${\rm card}(A,f,\omega)$ and ${\rm card}'(A,f,\omega)$ are $\Sigma$-measurable, $A(f,\omega)$ takes only countably many values and is  $\Sigma$-to-Borel-measurable, hence $A\in \mathcal{A}_{\rm meas}^{\rm ran}(\mathcal{P},\mathcal{R})$.

An access restriction is called bit restriction, if  
\begin{equation}
\label{D1}
|K'|=2,\quad \Lambda'=\{ \xi_j \colon j \in {\mathbb{N}}\}
\end{equation}
with
$\xi_j \colon \Omega \to K'=\{u_0,u_1\}$  an independent sequence of random variables  such that
\begin{equation}
\label{D2}
P(\{\xi_j = u_0\}) = P(\{\xi_j=u_1\}) = 1/2, \quad (j \in {\mathbb{N}}).
\end{equation}
The corresponding restricted randomized algorithms are called bit Monte Carlo algorithms. A non-adaptive version of these was considered in \cite{HNP04,NP04,GYW06,YH08}.

Most frequently used is the case of uniform distributions on $[0,1]$. This means  $K'=[0,1]$ and $\Lambda'=\{\eta_j: j \in {\mathbb{N}}\}$, with $(\eta_j)$ being independent uniformly dis\-tri\-bu\-ted on $[0,1]$ random variables over $(\Omega,\Sigma,{\mathbb P})$.

We also use the notion of a deterministic and of an (unrestricted) randomized algorithm and the corresponding notions of minimal errors. For this we refer to 
\cite{Hei05a, Hei05b}, as well as to Section 2 of \cite{Hei20}. Let us however mention that the definition of a deterministic algorithm follows a similar scheme as the one given above. 
More than that, we can give an equivalent definition of a deterministic algorithm, viewing it as a special case of a randomized algorithm with an arbitrary restriction $\mathcal{R}$. 
Namely, a deterministic algorithm is an $\mathcal R$-restricted randomized algorithm $A$ with
$$L_1\in\Lambda, \quad
L_{i+1}(K^i)\subseteq\Lambda \quad (i\in {\mathbb{N}}).
$$
Consequently, for each $f\in F$ and $\omega,\omega_1\in\Omega$ we have ${\rm card}_{\Lambda'}(A,f,\omega)=0$ and
\begin{eqnarray*}
	A(f)&:=&A(f,\omega)=A(f,\omega_1)\\
	{\rm card}(A,f)&:=&{\rm card}_{\bar{\Lambda}}(A,f,\omega)={\rm card}_\Lambda(A,f,\omega)={\rm card}_\Lambda(A,f,\omega_1).
\end{eqnarray*}
Thus, such an algorithm ignores $\mathcal{R}$ completely. 
For a deterministic algorithm $A$ relation \eqref{E3} turns into 
\begin{equation}
\label{E8}
e(\mathcal{P},A)=\sup_{f\in F}\|S(f)-A(f)\|_G.
\end{equation}
A deterministic algorithm is in $\mathcal{A}_{n,k}^{\rm ran}(\mathcal{P},\mathcal{R})$ iff $\sup_{f\in F}{\rm card}(A,f)\le n$. 
Taking the infimum in \eqref{A7} over all such $A$ gives the $n$-th minimal error in the deterministic setting $e_n^{\rm det}(\mathcal{P})$. Clearly,  $e(\mathcal{P},A)$ and $e_n^{\rm det}(\mathcal{P})$ do not depend on $\mathcal R$. It follows that for each restriction $\mathcal R$ and $n,k\in{\mathbb{N}}_0$
\begin{equation*}
e_{n,k}^{\rm ran } (\mathcal{P},\mathcal R)\le e_n^{\rm det } (\mathcal{P}).
\end{equation*}

A restricted randomized algorithm is a special case of an (unrestricted) randomized algorithm. Being intuitively clear, this was formally checked in \cite{Hei20}, Proposition 2.1 and Corollary 2.2. Moreover, it was shown there that for each restriction $\mathcal R$ and $n,k\in{\mathbb{N}}_0$
\begin{equation*}
e_n^{\rm ran } (\mathcal{P})\le e_{n,k}^{\rm ran } (\mathcal{P},\mathcal R),
\end{equation*}
where $e_n^{\rm ran}(\mathcal{P})$ denotes the $n$-th minimal error in the randomized setting,

\section{Deterministic vs.\ Restricted Randomized Algorithms}\label{sec:3}

In this section we derive a relation between minimal restricted randomized errors and minimal deterministic errors for general problems. 
Variants of 
the following result have been obtained for non-adaptive random bit 
algorithms  in \cite[Prop.~1]{HNP04}, and for adaptive algorithms
that ask for random bits and function values in alternating order in \cite{GHMR19b}.
Obviously, the latter does not permit to analyze a trade-off
between the number of random bits and the number of function
values to be used in a computation.

\begin{theorem}\label{thm2}
	For all problems  $\mathcal P=(F,G,S,K,\Lambda)$ and  probability spaces with finite access restriction $\mathcal R=\big((\Omega,\Sigma,{\mathbb P}),K',\Lambda'\big)$, see \eqref{B6}, and for all $n,k\in{\mathbb{N}}_0$ we have 
	\[
	e^{\rm ran}_{n,k}(\mathcal P,\mathcal R) \geq \frac{1}{3}  e^{\rm det}_{ 3n|K'|^{3k}}(\mathcal P).
	\]
\end{theorem}
Without loss of generality in the sequel we only consider access restrictions with the  property $K\cap K'=\emptyset$, thus $
\bar{K}=K\cup K'$, $\bar{\Lambda}=\Lambda\cup\Lambda'$. 

\begin{lemma}\label{lem:4}
	Let $n,k \in {\mathbb{N}}_0$, let $A$ be a randomized algorithm for $\mathcal{P}$ with access restriction $\mathcal R=\big((\Omega,\Sigma,{\mathbb P}),K',\Lambda'\big)$. For each $f\in F$ let 
	\begin{equation}\label{D8}
	B_f=\{ \omega \in \Omega \colon
	{\rm card}(A,f,\omega) \leq n,\ {\rm card}^\prime(A,f,\omega) \leq k\}.
	\end{equation}
	Then there is an  $\mathcal R$-restricted  randomized algorithm $\tilde{A}$ for  $\tilde{\mathcal P}=(F,\tilde{G},\tilde{S},\Lambda,K)$, where $\tilde{G}=G\oplus {\mathbb{R}}$ and $\tilde{S}=(S(f),0)$, satisfying for all $f\in F$ and $\omega\in \Omega$
	\begin{eqnarray}
	{{\rm card}}(\tilde{A},f,\omega)&\le&n
	\label{D9a}\\
	{{\rm card}}'(\tilde{A},f,\omega)&\le&k
	\label{E0a}\\
	\tilde{A}(f,\omega)&=&(A(f,\omega) \cdot 1_{B_f}(\omega),1_{B_f}(\omega)).\label{E1a}
	\end{eqnarray}
\end{lemma}
\begin{proof}
	Let $A= ((L_i)_{i=1}^\infty, (\tau_i)_{i=0}^\infty, (\varphi_i)_{i=0}^\infty)$.
	For $i\in {\mathbb{N}}_0$ and $a=(a_1,\dots,a_i)\in \overline{K}^i$ let 
	\begin{eqnarray*}
		d_{i+1}(a)&=& \left|\left\{L_1,L_2(a_1),\dots,L_{i+1}(a_1,\dots,a_i)\right\}\cap \Lambda\right|
		\\
		d_{i+1}'(a)&=&\left|\left\{L_1,L_2(a_1),\dots,L_{i+1}(a_1,\dots,a_i)\right\}\cap \Lambda'\right|
		\\
		\zeta_i(a)&=&\left\{\begin{array}{lll}
			1 & \mbox{if}& (d_{i+1}(a)>n) \vee (d_{i+1}'(a)>k) \\
			0  && \mbox{otherwise.}
		\end{array}
		\right. 
	\end{eqnarray*}
	Now we define  $\tilde{A}=((L_i)_{i=1}^\infty, (\tilde{\tau}_i)_{i=0}^\infty, (\tilde{\varphi}_i)_{i=0}^\infty)$ by setting for $i\in {\mathbb{N}}_0$ and $a\in \overline{K}^i$
	\begin{eqnarray*}
		\tilde{\tau}_i(a)&=&\max(\tau_i(a), \zeta_i(a))
		\\
		\tilde{\varphi}_i(a)&=&\left\{\begin{array}{lll}
			(\varphi_i(a),1) & \quad\mbox{if}\quad \zeta_i(a)\le \tau_i(a) \\
			(0,0) & \quad\mbox{if}\quad \zeta_i(a)>\tau_i(a).
		\end{array}
		\right. 
	\end{eqnarray*}
	To show \eqref{D9a}--\eqref{E1a} we fix $f\in F$, $\omega\in \Omega$ and define 
	\begin{eqnarray*}
		a_1= L_1(f,\omega),\quad &&a_i= (L_i(a_1,\dots,a_{i-1}))(f,\omega) \quad(i\ge 2).
	\end{eqnarray*}
	Let  $m=\overline{{\rm card}}(A,f,\omega)$ and let $q$ be the smallest number $q\in {\mathbb{N}}_0$ with $\zeta_q(a_1,\dots,a_q)=1$. 
	First assume that $\omega\in B_f$. 
	Then for all $i<m$
	\begin{equation*}
	(d_{i+1}(a_1,\dots,a_i)\le n) \wedge (d_{i+1}'(a_1,\dots,a_i)\le k),
	\end{equation*}
	thus $\zeta_i(a_1,\dots, a_i)=0$ and therefore $\tilde{\tau}_i(a_1,\dots,a_i)=0$. Furthermore, 
	$$
	\zeta_i(a_1,\dots, a_m)\le \tau_m(a_1,\dots, a_m)=1, 
	$$
	which means $\overline{{\rm card}}(\tilde{A},f,\omega)=m$,
	\begin{eqnarray*}
		{{\rm card}}(\tilde{A},f,\omega)&=&d_m(a_1,\dots,a_{m-1})\le n
		\\
		{{\rm card}}'(\tilde{A},f,\omega)&=& d_m'(a_1,\dots,a_{m-1})\le k
		\\
		\tilde{A}(f,\omega)&=&(\varphi_m(a_1,\dots,a_m),1)=(A(f,\omega),1).
	\end{eqnarray*}
	Now let $\omega\in\Omega\setminus B_f$, hence 
	\begin{eqnarray*}
		\tau_0=\tau_1(a_1)=\dots =\tau_q(a_1,\dots,a_q)=0
		\\
		(d_{q+1}(a_1,\dots,a_q)>n) \vee (d_{q+1}'(a_1,\dots,a_q)>k),
	\end{eqnarray*}
	thus $\tilde{\tau}_q(a_1,\dots,a_q)=1$.
	Consequently, 
	\begin{eqnarray*}
		{{\rm card}}(\tilde{A},f,\omega)&\le & d_q(a_1,\dots,a_{q-1})\le n
		\\
		{{\rm card}}'(\tilde{A},f,\omega)&\le & d_q'(a_1,\dots,a_{q-1})\le  k
		\\
		\tilde{A}(f,\omega)&=&(0,0).
	\end{eqnarray*}
\end{proof}

The key ingredient of the proof of Theorem \ref{thm2} is the following 
\begin{lemma}\label{l3}
	Let $n,k \in {\mathbb{N}}_0$ and let $A$ be a randomized algorithm for $\mathcal{P}$ with finite access restriction $\mathcal R=\big((\Omega,\Sigma,{\mathbb P}),K',\Lambda'\big)$ such that  
	\begin{equation}\label{eq6}
	{\rm card} (A,f,\omega) \leq n ,\quad {\rm card}^\prime(A,f,\omega) \leq k
	\end{equation}
	for all $f \in F$ and $\omega \in \Omega$.
	Then  there exists a deterministic algorithm $A^*$ for $\mathcal P$ with
	\begin{equation}
	\label{E9}
	A^*(f) = {\mathbb E}\,(A(f,\cdot)), \quad {\rm card}(A^*,f) \leq n  |K'|^k\quad (f \in F).
	\end{equation}
\end{lemma}
\begin{proof}
	Let
	$\mathcal P=(F,G,S,K,\Lambda)$, $A = ((L_i)_{i=1}^\infty, (\tau_i)_{i=0}^\infty, (\varphi_i)_{i=0}^\infty)$.
	We argue by induction over $m=n+k$. 
	If $m=0$, then $\tau_0=1$, hence $\overline{{\rm card}}(A,f,\omega)=0$, thus $A(f,\omega)=\varphi_0$ for all $f \in F$ and $\omega \in \Omega$, and the result follows. 
	
	Now let $m\ge 1$. We can assume that $\tau_0=0$, otherwise $A$ satisfies \eqref{eq6} with $n=k=0$ and we are back to the case $m=0$. Let $\tilde{K}\subset \overline{K}$ be defined by
	\begin{equation*}
	\tilde{K}=\left\{ \begin{array}{lll}
	\left\{u\in K:\,L_1^{-1}(\{u\})\ne\emptyset\right\}& \quad\mbox{if}\quad L_1\in \Lambda   \\[.2cm]
	\left\{u\in K':\,{\mathbb P}(L_1^{-1}(\{u\}))\ne0\right\} & \quad\mbox{if}\quad L_1\in \Lambda'.  
	\end{array}
	\right. 
	\end{equation*}
	For every $u\in\tilde{K}$ we define a problem $\mathcal P_u=(F_u,G,S_u,K,\Lambda_u)$ and a probability space with access restriction   $\mathcal R_u=\big((\Omega_u,\Sigma_u,{\mathbb P}_u),K',\Lambda'_u\big)$ as follows. 
	If $L_1\in \Lambda$, we set $\mathcal R_u=\mathcal R$ and
	\begin{eqnarray*}
		F_u&=&\{f\in F:\, L_1(f)=u\}, \quad
		S_u=S|_{F_u},\quad
		\Lambda_u=\{\lambda|_{F_u}:\, \lambda\in\Lambda\}.
	\end{eqnarray*}
	If $L_1\in\Lambda'$, we put $\mathcal P_u=\mathcal P$ and 
	\begin{eqnarray*}
		\Omega_u&=&\{\omega\in \Omega:\, L_1(\omega)=u\},\quad
		\Sigma_u=\{B\cap \Omega_u:\;B\in \Sigma\}
		\\
		{\mathbb P}_u(C)&=& {\mathbb P}\left(\Omega_u\right)^{-1}{\mathbb P}\left(C\right)\quad(C\in \Sigma_u),\quad \Lambda_u'=\{\lambda'|_{\Omega_u}:\, \lambda\in\Lambda'\}.
	\end{eqnarray*}
	Let $\varrho_u:\Lambda\cup\Lambda'\to \Lambda_u\cup\Lambda'_u$ be defined as 
	\begin{equation*}
	\varrho_u(\lambda)=\left\{\begin{array}{lll}
	\lambda|_{F_u}  & \quad\mbox{if}\quad \lambda\in \Lambda   \\
	\lambda|_{\Omega_u}  & \quad\mbox{if}\quad \lambda\in \Lambda'  
	\end{array}
	\right. 
	\end{equation*}
	and let $\sigma_u: \Lambda_u\cup\Lambda'_u\to \Lambda\cup\Lambda'$ be any mapping satisfying
	\begin{equation}
	\label{C6}
	\varrho_u\circ \sigma_u={\rm id}_{\Lambda_u\cup\Lambda'_u}.
	\end{equation}
	Furthermore, we define a random algorithm  $A_u=((L_{i,u})_{i=1}^\infty, (\tau_{i,u})_{i=0}^\infty, (\varphi_{i,u})_{i=0}^\infty)$ for $\mathcal P_u$ with access restriction $\mathcal R_u$ by setting
	for $i\ge 0$, $z_1,\dots,z_i\in \overline{K}$
	\begin{eqnarray}
	L_{i+1,u}(z_1,\dots,z_i)&= & 
	\big( \varrho_u\circ L_{i+2}\big)(u,z_1,\dots,z_i)
	\label{P1}\\[.1cm]
	\tau_{i,u}(z_1,\dots,z_i)&=&\tau_{i+1}(u,z_1,\dots,z_i)
	\label{P2}\\[.1cm]
	\varphi_{i,u}(z_1,\dots,z_i)&=&\varphi_{i+1}(u,z_1,\dots,z_i)
	\label{P3}
	\end{eqnarray}
	(in this and similar situations below the case $i=0$ with variables $z_1,\dots,z_i$ is understood in the obvious way: no dependence on $z_1,\dots,z_i$).
	
	Next we establish the relation of the algorithms $A_u$ to $A$. Fix $f \in F_u$, $\omega \in \Omega_u$, and let $(a_i)_{i=1}^\infty\subseteq \overline{K}$ be given by 
	\begin{eqnarray}
	a_1&=&L_1(f,\omega)=u\label{RC1}\\
	a_i&=&\big(L_i(a_1,\dots,a_{i-1})\big)(f,\omega)\quad(i\ge 2)\label{N6},
	\end{eqnarray}
	and similarly $(a_{i,u})_{i=1}^\infty\subseteq \overline{K}$ by
	\begin{eqnarray}
	a_{i,u}&=&\big(L_{i,u}(a_{1,u},\dots,a_{i-1,u})\big)(f,\omega)\label{N9}.
	\end{eqnarray}
	We show by induction that 
	\begin{equation}
	\label{H3}
	a_{i,u}=a_{i+1}\quad (i\in{\mathbb{N}}). 
	\end{equation}
	Let $i=1$. Then \eqref{N9}, \eqref{P1}, \eqref{RC1}, and  \eqref{N6} imply 
	\begin{eqnarray*}
		a_{1,u}&=&L_{1,u}(f,\omega)=\big(L_2(u)\big)(f,\omega)=\big(L_2(a_1)\big)(f,\omega)=a_2.
	\end{eqnarray*}
	For the induction step we let  $j\in {\mathbb{N}}$ and suppose that \eqref{H3} holds for all  $i\le j$. 
	Then \eqref{N9}, \eqref{P1},   \eqref{H3}, and \eqref{N6} yield
	\begin{eqnarray*}
		a_{j+1,u}
		&=& (L_{j+1,u}(a_{1,u},\dots,a_{j,u}))(f,\omega)
		= (L_{j+2}(u,a_{1,u},\dots,a_{j,u}))(f,\omega)
		\\
		&=& (L_{j+2}(a_1,a_2,\dots,a_{j+1}))(f,\omega)
		=a_{j+2}.
	\end{eqnarray*}
	This proves \eqref{H3}. As a consequence of this relation and of \eqref{P1}, \eqref{P2}, and  \eqref{P3} we obtain 
	for all $i\in{\mathbb{N}}_0$ 
	\begin{eqnarray*}
		L_{i+1,u}(a_{1,u},\dots,a_{i,u})&=&\big( \varrho_u\circ L_{i+2}\big)(u,a_{1,u},\dots,a_{i,u})=\big( \varrho_u\circ L_{i+2}\big)(a_1,\dots,a_{i+1})
		\\
		\tau_{i,u}(a_{1,u},\dots,a_{i,u})&=&\tau_{i+1}(u,a_{1,u},\dots,a_{i,u})=\tau_{i+1}(a_1,\dots,a_{i+1})
		\\
		\varphi_{i,u}(a_{1,u},\dots,a_{i,u})&=&\varphi_{i+1}(u,a_{1,u},\dots,a_{i,u})=\varphi_{i+1}(a_1,\dots,a_{i+1}).
	\end{eqnarray*}
	Hence, for all $f\in F_u$ and $\omega \in \Omega_u$
	\begin{eqnarray}
	\overline{{\rm card}}(A_u,f,\omega)&=&\overline{{\rm card}}(A,f,\omega)-1
	\notag\\
	A_u(f,\omega)&=&A(f,\omega).\label{A4}
	\end{eqnarray}
	Furthermore, if $L_1\in\Lambda$, then 
	\begin{eqnarray*}
		{\rm card}(A_u,f,\omega)&=&{\rm card}(A,f,\omega)-1\le n-1
		\\
		 {\rm card}'(A_u,f,\omega)&=&{\rm card}'(A,f,\omega)\le k, 
	\end{eqnarray*}
	and if $L_1\in\Lambda'$,
	\begin{eqnarray*}
		{\rm card}(A_u,f,\omega)&=&{\rm card}(A,f,\omega)\le n
		\\
		{\rm card}'(A_u,f,\omega)&=&{\rm card}'(A,f,\omega)-1\le k-1. 
	\end{eqnarray*}

	Now we apply the induction assumption and obtain a deterministic algorithm $$A_u^*=((L_{i,u}^*)_{i=1}^\infty, (\tau_{i,u}^*)_{i=0}^\infty, (\varphi_{i,u}^*)_{i=0}^\infty)$$ for $\mathcal P_u$ with 
	\begin{equation}
	\label{A5}
	A_u^*(f) = {\mathbb E}\,_{{\mathbb P}_u}(A_u(f,\cdot)) 
	\end{equation}
	and 
	\begin{equation}
	\label{B3}
	{\rm card}(A^*_u,f) \leq\left\{\begin{array}{lll}
	(n-1)  |K'|^k  & \quad\mbox{if}\quad L_1\in \Lambda   \\
	n  |K'|^{k-1}  & \quad\mbox{if}\quad L_1\in \Lambda'      
	\end{array}
	\right. 
	\end{equation}
	for every $f \in F_u$. 
	
	Finally we use the algorithms $A_u^*$ to compose a deterministic algorithm  
	$$
	A^* = ((L_i^*)_{i=1}^\infty, (\tau_i^*)_{i=0}^\infty,
	(\varphi_i^*)_{i=0}^\infty)
	$$ for $\mathcal P$. This and the completion of the proof is done  separately for each of the cases 
	$L_1\in\Lambda$ and $L_1\in\Lambda'$.
	
	If $L_1\in\Lambda$, then we set
	\begin{eqnarray*}
		L_1^*=L_1,\quad\tau_0^*=\tau_0=0, \quad \varphi_0^*=\varphi_0,\label{M1}
	\end{eqnarray*}
	furthermore, for $i\in{\mathbb{N}}$, $z_1\in \tilde{K}$, $z_2,\dots,z_i\in \overline{K}$ we let (with $\sigma_{z_1}$ defined by \eqref{C6})
	\begin{eqnarray}
	L_{i+1}^*(z_1,\dots,z_i)&= & 
	\big(\sigma_{z_1}\circ L_{i,z_1}^*\big)(z_2,\dots,z_i)
	\label{Q1}\\[.1cm]
	\tau_{i}^*(z_1,\dots,z_i)&=&\tau_{i-1,z_1}^*(z_2,\dots,z_i)
	\label{Q2}\\[.1cm]
	\varphi_{i}^*(z_1,\dots,z_i)&=&\varphi_{i-1,z_1}^*(z_2,\dots,z_i).
	\label{Q3}
	\end{eqnarray}
	For $i\ge 1$, $z_1\in \overline{K}\setminus \tilde{K}$, and $z_2,\dots,z_i\in \overline{K}$ we define 
	\begin{eqnarray*}
		L_{i+1}^*(z_1,\dots,z_i)=  L_1,\quad
		\tau_{i}^*(z_1,\dots,z_i)=1, \quad
		\varphi_{i}^*(z_1,\dots,z_i)=\varphi_0.
	\end{eqnarray*}
	Let $u\in\tilde{K}$ and $f\in F_u$. We show that
	\begin{eqnarray}
	A^*(f)&=&A_u^*(f)
	\label{B4}\\
	{\rm card}(A^*,f)&=&{\rm card}(A_u^*,f)+1.\label{B5}
	\end{eqnarray}
	Let $(b_i)_{i=1}^\infty\subseteq \overline{K}$ be given by 
	\begin{eqnarray}
	b_1&=&L_1^*(f)=L_1(f)=u   
	\label{G7}\\
	b_i&=&\big(L_i^*(b_1,\dots,b_{i-1})\big)(f)\quad(i\ge 2)\label{A8},
	\end{eqnarray}
	and similarly $(b_{i,u})_{i=1}^\infty\subseteq \overline{K}$ by
	\begin{eqnarray}
	b_{i,u}&=&\big(L_{i,u}^*(b_{1,u},\dots,b_{i-1,u})\big)(f)\label{A9}.
	\end{eqnarray}
	Then
	\begin{equation}
	\label{B2}
	b_{i+1}=b_{i,u}\quad (i\in{\mathbb{N}}). 
	\end{equation}
	Indeed, 
	for $i=1$ we conclude from  \eqref{A8}, \eqref{G7}, \eqref{Q1}, and \eqref{A9} 
	\begin{eqnarray*}
		b_2&=&(L^*_2(b_1))(f)=(L^*_2(u))(f)=L^*_{1,u}(f)=b_{1,u}.
	\end{eqnarray*}
	Now let  $j\in {\mathbb{N}}$ and assume \eqref{B2} holds for all  $i\le j$. 
	By \eqref{A8},  \eqref{G7},  \eqref{Q1},  and \eqref{A9} 
	\begin{eqnarray*}
		b_{j+2}
		&=& (L^*_{j+2}(b_1,b_2,\dots,b_{j+1}))(f)
		= (L^*_{j+2}(u,b_{1,u},\dots,b_{j,u}))(f)
		\\
		&=& (L^*_{j+1,u}(b_{1,u},\dots,b_{j,u}))(f)
		=b_{j+1,u}.
	\end{eqnarray*}
	This proves \eqref{B2}.
	It follows from \eqref{B2},  \eqref{G7},  \eqref{Q2},  and  \eqref{Q3} that
	for all $i\in{\mathbb{N}}_0$ 
	\begin{eqnarray*}
		\tau_{i+1}^*(b_1,\dots,b_{i+1})&=&\tau_{i+1}^*(u,b_{1,u},\dots,b_{i,u})=\tau_{i,u}^*(b_{1,u},\dots,b_{i,u})
		\\
		\varphi_{i+1}^*(b_1,\dots,b_{i+1})&=&\varphi_{i+1}^*(u,b_{1,u},\dots,b_{i,u})=\varphi_{i,u}^*(b_{1,u},\dots,b_{i,u}).
	\end{eqnarray*}
	This shows \eqref{B4} and \eqref{B5}. From \eqref{B4}, \eqref{A5}, and \eqref{A4} we conclude for $u\in \tilde{K}$, $f\in F_u$, recalling that $\mathcal R_u=\mathcal R$,
	$$
	A^*(f)=A_u^*(f)= {\mathbb E}\,_{{\mathbb P}}(A_u(f,\cdot)) ={\mathbb E}\,_{{\mathbb P}}(A(f,\cdot)). 
	$$
	Since $\cup_{u\in\tilde{K}}F_u=F$, the first relation of \eqref{E9} follows. The second relation is a direct consequence of \eqref{B5} and \eqref{B3},  completing the induction for the case $L_1\in\Lambda$.

	If $L_1\in\Lambda'$, then we use the algorithms $(A_u^*)_{u\in \tilde{K}}$ for $\mathcal P_u=\mathcal P$ and Lemma 3 of \cite{Hei05b} to obtain a deterministic algorithm $A^*$ for $\mathcal P$ such that for $f\in F$
	\begin{eqnarray}
	A^*(f)&=&\sum_{u\in  \tilde{K}}{\mathbb P}(L_1^{-1}(\{u\})A_u^*(f)
	\label{C0}\\
	{\rm card}(A^*,f)&=&\sum_{u\in \tilde{K}}{\rm card}(A_u^*,f).\label{C1}
	\end{eqnarray}
	It follows from  \eqref{C0},  \eqref{A5}, and \eqref{A4} that
	\begin{eqnarray*}
		A^*(f)&=&\sum_{u\in K':\,{\mathbb P}(L_1^{-1}(\{u\}))>0 }{\mathbb P}(L_1^{-1}(\{u\}){\mathbb E}\,_{{\mathbb P}_u}A_u(f,\cdot)
		\\
		&=&\sum_{u\in K':\,{\mathbb P}(L_1^{-1}(\{u\}))>0 }\int_{L_1^{-1}(\{u\})}A_u(f,\omega) d{\mathbb P}(\omega)
		\\
		&=&\sum_{u\in K':\,{\mathbb P}(L_1^{-1}(\{u\}))>0 }\int_{L_1^{-1}(\{u\})}A(f,\omega) d{\mathbb P}(\omega)
		={\mathbb E}\,_{{\mathbb P}}A_u(f,\cdot).
	\end{eqnarray*}
Furthermore, \eqref{B3} and \eqref{C1} imply
	${\rm card}(A^*,f)\le n |K'|^k$.
\end{proof}

\begin{proof}
	\hspace{-0.1cm}\textit{\textbf{of Theorem \ref{thm2}}}
	The proof is similar to the proof of \cite[Lem.~11]{GHMR19b}.
	Let $\delta>0$ and let 
	$$
	A= ((L_i)_{i=1}^\infty, (\tau_i)_{i=0}^\infty, (\varphi_i)_{i=0}^\infty) \in \mathcal{A}^{\rm ran}_{n,k}(\mathcal P,\mathcal R)
	$$ 
	be a randomized algorithm for $\mathcal P$ with restriction $\mathcal R$ satisfying 
	\begin{equation}
	\label{E4}
	e(A,\mathcal P) \le 
	e^{\rm ran}_{n,k}(\mathcal P,\mathcal R)+\delta.
	\end{equation}
	For $f \in F$ define 
	\[
	B_f=\{ \omega \in \Omega \colon
	{\rm card}(A,f,\omega) \leq 3n,\ {\rm card}^\prime(A,f,\omega) \leq 3k\}.
	\]
	Observe that $B_f \in \Sigma$ and $P(B_f)\geq 1/3$.
	For the conditional expectation
	\begin{equation*}
	{\mathbb E}\,(A(f,\cdot) \,|\, B_f) = 
	\frac{{\mathbb E}\,\left(A(f,\cdot) \cdot 1_{B_f}\right)}{P(B_f)} 
	\end{equation*}
	of $A(f,\cdot)$ given $B_f$ we obtain
	\begin{eqnarray}
	\lefteqn{3{\mathbb E}\,\left\|S(f)-A(f,\cdot)\right\|_G} \nonumber\\
	&\geq& 
	{\mathbb E}\,\left(\|S(f)-A(f,\cdot)\|_G \,|\, B_f\right)
	\geq \left\|S(f)-{\mathbb E}\,\left(A(f,\cdot) \,|\, B_f \right)\right\|_G
	\label{E5}
	\end{eqnarray}
	by means of Jensen's inequality. Our goal is now to design a deterministic algorithm with  input-output mapping  $f \mapsto {\mathbb E}\,(A(f,\cdot) \,|\, B_f)$.

	From Lemma \ref{lem:4} we conclude that there is an $\mathcal R$-restricted  randomized algorithm $\tilde{A}=((L_i)_{i=1}^\infty, (\tilde{\tau}_i)_{i=0}^\infty, (\tilde{\varphi}_i)_{i=0}^\infty)$ for $\tilde{\mathcal P}=(F,\tilde{G},\tilde{S},\Lambda,K)$, where $\tilde{G}=G\oplus {\mathbb{R}}$ and   $\tilde{S}(f)=(S(f),0)$ $(f\in F)$, satisfying for all $f\in F$ and $\omega\in \Omega$
	\begin{eqnarray*}
		&&{{\rm card}}(\tilde{A},f,\omega)\le 3n, \quad
		{{\rm card}}'(\tilde{A},f,\omega)\le 3k,
		\label{E0}\\
		&&\tilde{A}(f,\omega)=(A(f,\omega) \cdot 1_{B_f}(\omega),1_{B_f}(\omega)).\label{E1}
	\end{eqnarray*}
	By Lemma \ref{l3} there is a deterministic algorithm $A^*= ((L_i^*)_{i=1}^\infty, (\tau_i^*)_{i=0}^\infty, (\varphi_i^*)_{i=0}^\infty)$
	for $\tilde{\mathcal P}$
	such that for all $f \in F$ 
	\begin{eqnarray*}
		{\rm card}(A^*,f) \leq 3n  |K'|^{3k},\quad
		A^*(f)=\left(\int_{B_f}A(f,\omega) d{\mathbb P}(\omega),{\mathbb P}(B_f)\right).
	\end{eqnarray*}
	It remains to modify $A^*$ as follows 
	$$
	\tilde{A}^*= ((L_i^*)_{i=1}^\infty, (\tau_i^*)_{i=0}^\infty, (\psi_i^*)_{i=0}^\infty),
	$$
	where for $i\in{\mathbb{N}}_0$ and $a\in K^i$
	\begin{equation*}
	\psi_i^*(a)=  \left\{\begin{array}{lll}
	\frac{\varphi_{i,1}^*(a)}{\varphi_{i,2}^*(a)} & \quad\mbox{if}\quad  \varphi_{i,2}^*(a)\ne 0  \\
	0 & \quad\mbox{if}\quad\varphi_{i,2}^*(a)= 0,
	\end{array}
	\right. 
	\end{equation*}
	with $\varphi_i^*(a)=(\varphi_{i,1}^*(a),\varphi_{i,2}^*(a))$ being the splitting into the $G$ and ${\mathbb{R}}$ component. Hence for each $f\in F$
\begin{eqnarray*}
{\rm card}(\tilde{A}^*,f) &\leq& 3n  |K'|^{3k}
\\
\tilde{A}^*(f)&=&{\mathbb E}\,(A(f,\cdot) \,|\, B_f),
\end{eqnarray*}
and therefore we conclude, using \eqref{E4} and \eqref{E5},
\begin{eqnarray*}
	e^{\rm det}_{ 3n|K'|^{3k}}(\mathcal P)\le e(\tilde{A}^*,\tilde{\mathcal P})\le 3e(A,\mathcal P)\le 3(e^{\rm ran}_{n,k}(\mathcal P,\mathcal R)+\delta)
\end{eqnarray*}
for each $\delta>0$.
\end{proof}

\section{Applications}\label{sec:4}
\subsection {Integration of functions in Sobolev spaces}
Let $r,d\in{\mathbb{N}}$, $1\le p<\infty$, $Q=[0,1]^d$, let $C(Q)$ be the space of continuous functions on $Q$, and $W_p^r(Q)$ the Sobolev space, see \cite{Ada75}. Then $W_p^r(Q)$ is embedded into  $C(Q)$ iff
\begin{equation}
\label{E2}
\begin{array}{lllll}
(p=1   \;\mbox{and} \quad r/d\ge 1)\quad    
\mbox{or}\quad
(1<p<\infty  \;\mbox{and} \quad  r/d>1/p). 
\end{array} 
\end{equation}
Let $B_{W_p^r(Q)}$ be the unit ball of $W_p^r(Q)$, $B_{W_p^r(Q)}\cap C(Q)$ the set of those elements of the unit ball which are continuous (more precisely, of equivalence classes, which contain a continuous representative), and define
$$
F_1= \left\{\begin{array}{lll}
& B_{W_p^r(Q)}&\quad\mbox{if the embedding condition \eqref{E2} holds}   \\
& B_{W_p^r(Q)}\cap C(Q)&\quad\mbox{otherwise.}    
\end{array}
\right. 
$$
Moreover, let $I_1:W_p^r(Q)\to {\mathbb{R}}$ be the integration operator
$$
I_1f=\int_Q f(x)dx.
$$
and let $\Lambda_1=\{\delta_x\colon x\in Q\}$ be the set of point evaluations, where $\delta_x(f)=f(x)$.
Put into the general framework of \eqref{I4}, we consider the problem
$
\mathcal{P}_1= (F_1,{\mathbb{R}},I_1,{\mathbb{R}},\Lambda_1).
$
Set $\bar{p}=\min(p,2)$. 
Then the following is known (for (\ref{AB4}--\ref{BB2}) below see \cite{Hei12} and references therein). There are constants $c_{1-6}>0$ such that for all $n\in{\mathbb{N}}_0$
\begin{equation}
\label{AB4}
c_1 n^{-r/d-1+1/\bar{p}}\le e_n^{\rm ran}(\mathcal{P}_1)
\le c_2 n^{-r/d-1+1/\bar{p}},
\end{equation}
moreover, if the embedding condition holds, then 
\begin{equation}
\label{BB1}
c_3n^{-r/d}\le e_n^{\rm det}(\mathcal{P}_1)
\le c_4 n^{-r/d},
\end{equation}
while if the embedding condition does not hold, then 
\begin{equation}
\label{BB2}
c_5\le e_n^{\rm det}(\mathcal{P}_1)\le c_6.
\end{equation}

Theorem \ref{thm2} immediately gives (compare this with the rate in the unrestricted setting \eqref{AB4})
\begin{corollary}\label{cor:1} Assume that the embedding condition \eqref{E2} does not hold and  let
	$\mathcal{R}$ be any finite access restriction, see  \eqref{B6}.  Then there is a constant $c>0$ such that for all $n,k\in{\mathbb{N}}$ 
	$$
	e_{n,k}^{\rm ran} (\mathcal{P}_1,\mathcal{R})\ge c.
	$$
\end{corollary}

It was shown in \cite{HNP04}, that if the embedding condition holds, then $(2 + d) \log_2 n$ random bits suffice to reach the rate of the unrestricted randomized setting, thus, if $\mathcal{R}$ is a bit restriction (see \eqref{D1}--\eqref{D2}), then there are constants $c_1,c_2>0$ such that for all $n\in{\mathbb{N}}$
\begin{equation}
\label{AB5}
c_1 n^{-r/d-1+1/\bar{p}}\le e_n^{\rm ran} (\mathcal{P}_1)\le  e_{n, (2 + d) \log_2 n}^{\rm ran} (\mathcal{P}_1,\mathcal{R})
\le c_2 n^{-r/d-1+1/\bar{p}}.
\end{equation}
The following consequence of  Theorem \ref{thm2} shows that the number of random bits used in the (non-adaptive) algorithm from \cite{HNP04} giving \eqref{AB5} is optimal up to a constant factor, also for adaptive algorithms.
\begin{corollary}\label{cor:2} Assume that the embedding condition holds and let 
	$\mathcal{R}$ be any finite access restriction. Then for each $\sigma$ with $0<\sigma\le 1-1/\bar{p}$ and each $c_0>0$ there are constants $c_1>0$, $c_2\in {\mathbb{R}}$ such that for all $n,k\in{\mathbb{N}}$ 
	$$
	e_{n,k}^{\rm ran} (\mathcal{P}_1,\mathcal{R})\le c_0n^{-r/d-\sigma}.
	$$
	implies 
	$$
	k\ge c_1\sigma\log_2 n+c_2.
	$$
\end{corollary}
\begin{proof}
	Let $\mathcal{R}=\big((\Omega,\Sigma,{\mathbb P}),K',\Lambda'\big)$. By Theorem  \ref{thm2} and \eqref{BB1},
	\begin{eqnarray*}
		c_0n^{-r/d-\sigma}&\ge& e^{\rm ran}_{n,k}(\mathcal{P}_1,\mathcal{R}) \geq 3^{-1}  e^{\rm det}_{ 3n|K'|^{3k}}(\mathcal{P}_1)
		\ge 3^{-1}c_3(n|K'|^{3k})^{-r/d},	
	\end{eqnarray*}
	implying
	\begin{eqnarray*}
		\log_2 c_0-	\sigma\log_2 n&\ge& \log_2 (c_3/3)- \frac{3kr}{d}\log_2 |K'|,	
	\end{eqnarray*}
	thus, 
	\begin{eqnarray*}
		k&\ge& \frac{d}{3r\log_2 |K'|}(\sigma\log_2 n-\log_2 c_0+\log_2 (c_3/3)	).	
	\end{eqnarray*}
\end{proof}
\subsection {Integration of Lipschitz functions over the Wiener space}
Let $\mu$ be the Wiener measure on $C([0,1])$,
\begin{equation*}
F_2=\{f: C([0,1])\to {\mathbb{R}},\; |f(x)-f(y)|\le \|x-y\|_{ C([0,1])}\quad(x,y \in  C([0,1]))\},
\end{equation*}
$G={\mathbb{R}}$, let $I_2:F\to {\mathbb{R}}$ be the integration operator given by
$$
I_2f=\int_{C([0,1])} f(x)d\mu(x),
$$
and $\Lambda_2=\{\delta_x\colon x\in C([0,1])\}$,
so we consider the problem
$
\mathcal{P}_2=(F_2,{\mathbb{R}},I_2,{\mathbb{R}},\Lambda_2).
$
There exist constants $c_{1-4} >0$ such that 
\begin{equation}\label{eq1}
c_1  n^{-1/2} (\log_2 n)^{-3/2} \leq 
e^{{\rm ran}}_n(\mathcal{P}_2) \leq c_2  n^{-1/2} (\log_2 n)^{-1/2}
\end{equation}
and
\begin{equation}\label{eq3}
c_3  (\log_2 n)^{-1/2} \leq e^{{\rm det}}_n(\mathcal{P}_2)
\leq c_4  (\log_2 n)^{-1/2}
\end{equation}
for every $n \geq 2$,
see \cite{CDMGR09}, Theorem 1 and Proposition 3 for \eqref{eq3} and 
Theorems 11 and  12 for \eqref{eq1}. Moreover, it is shown in \cite{GHMR19b}, Theorem 8 and Remark 9, that if $\mathcal{R}$ is a bit restriction, then 
there exist a constants $c_1>0$, $c_2\in{\mathbb{N}}$ such that for all $n\in{\mathbb{N}}$ with $n\ge 3$
\begin{equation}\label{eq4}
e^{\rm ran}_{n,\kappa(n)}(\mathcal{P}_2,\mathcal{R}) \leq c_1  n^{-1/2} (\log_2 n)^{3/2},
\end{equation}
where
\begin{equation}
\label{E6}
\kappa(n) = c_2\lceil n  (\log_2 n)^{-1}  \log_2 (\log_2 n) \rceil.
\end{equation}
Our results imply that the number of random bits \eqref{E6} used in the algorithm of \cite{GHMR19b} giving the upper bound in \eqref{eq4}  is optimal (up to $\log$ terms) in the following sense. 
\begin{corollary}\label{cor:3}  Let $\mathcal{R}$ be a finite access restriction. For each $\alpha\in{\mathbb{R}}$ and each $c_0>0$ there are constants $c_1>0$ and $c_2\in{\mathbb{R}}$ such that for all $n,k\in{\mathbb{N}}$ with $n\ge 2$
	$$
	e_{n,k}^{\rm ran} (\mathcal{P}_2,\mathcal{R})\le c_0n^{-1/2}(\log_2 n)^\alpha.
	$$
	implies 
	\begin{equation}
	\label{E7}
	k\ge c_1n(\log_2 n)^{-2\alpha}+c_2.
	\end{equation}
\end{corollary}
\begin{proof}
	Let $\mathcal{R}=\big((\Omega,\Sigma,{\mathbb P}),K',\Lambda'\big)$.	We use Theorem  \ref{thm2} again.  From \eqref{eq3} we obtain
	\begin{eqnarray*}
		c_0n^{-1/2}(\log_2 n)^\alpha&\ge& e^{\rm ran}_{n,k}(\mathcal{P}_2,\mathcal{R}) \geq 3^{-1}  e^{\rm det}_{ 3n|K'|^{3k}}(\mathcal{P}_2)
		\ge 3^{-1}c_3\log_2(3n|K'|^{3k})^{-1/2},	
	\end{eqnarray*}
	thus
	\begin{eqnarray*}
		\log_2(3n)+3k\log_2|K'|\ge \frac{c_3^2}{9c_0^2}n(\log_2 n)^{-2\alpha},
	\end{eqnarray*}
	which implies 
	\begin{eqnarray}
	k\ge (3\log_2|K'|)^{-1}\left(\frac{c_3^2}{9c_0^2}n(\log_2 n)^{-2\alpha}-\log_2(3n)\right).\label{F0}
	\end{eqnarray}
	Choosing $n_0\in{\mathbb{N}}$ in such a way that for $n\ge n_0$ 
	$$
	\frac{c_3^2}{18c_0^2}n(\log_2 n)^{-2\alpha}\ge \log_2(3n)
	$$
	leads to
	\begin{eqnarray*}
		k\ge (3\log_2|K'|)^{-1}\left(\frac{c_3^2}{18c_0^2}n(\log_2 n)^{-2\alpha}-\log_2(3n_0)\right).
	\end{eqnarray*}
\end{proof}

{\bf Acknowledgement}. The author thanks Mario Hefter and Klaus Ritter
for discussions on the subject of this paper.

\end{document}